\documentclass[11pt, reqno]{amsart}
\usepackage{amsthm}
\usepackage{amsmath,amssymb,amsthm,mathrsfs,amscd}
\usepackage{color}
\usepackage[all]{xy}
\usepackage[english]{babel}
\usepackage{graphicx}
\usepackage[latin1]{inputenc}
\usepackage[T1]{fontenc}
\usepackage[final=true]{hyperref}
\usepackage{tikz}
\usetikzlibrary{patterns}
\usetikzlibrary{arrows}
\usepackage{rotating}
\usepackage{cases}
\usepackage{xkeyval}
\usepackage{ytableau}

\theoremstyle{plain}
\newtheorem*{prop}{Proposition}
\newtheorem{thm}{Theorem}
\newtheorem*{lem}{Lemma}
\newtheorem*{cor}{Corollary}

\theoremstyle{definition}
\newtheorem*{ex}{Example}
\newtheorem*{defi}{Definition}

\newtheorem*{rem}{Remark}

\theoremstyle{remark}

\DeclareMathOperator{\Hom}{Hom}

\DeclareMathOperator{\modd}{mod}

\DeclareMathOperator{\Repp}{Rep}
\DeclareMathOperator{\In}{in}
\DeclareMathOperator{\out}{out}

\DeclareMathOperator{\Coker}{Coker}

\DeclareMathOperator{\Ir}{Irr}

\DeclareMathOperator{\wt}{wt}

\DeclareMathOperator{\N}{\mathbb{N}}

\DeclareMathOperator{\Ext}{Ext}

\DeclareMathOperator{\op}{op}

\newcommand{\ch}{\scriptscriptstyle\vee{}}
\newcommand{\kg}{\unlhd}
\newcommand{\Qu}{\mathrm{Q}}

\advance\textwidth by 1.2in  \advance\oddsidemargin by -.6in \advance\evensidemargin by -.6in

\newcounter{cnt}
 \makeatletter
\def\mydggeometry{\makeatletter\dg@YGRID=1\dg@XGRID=20\unitlength=0.003pt\makeatother}
\makeatother \theoremstyle{remark}

\numberwithin{equation}{section}
\makeatletter
\def\section{\def\@secnumfont{\mdseries}\@startsection{section}{1}%
  \z@{.7\linespacing\@plus\linespacing}{.5\linespacing}%
  {\normalfont\scshape\centering}}
\def\subsection{\def\@secnumfont{\bfseries}\@startsection{subsection}{2}%
  {\parindent}{.5\linespacing\@plus.7\linespacing}{-.5em}%
  {\normalfont\bfseries}}
\makeatother

\begin{document}
\title[Nakajima quiver varieties, affine crystals and combinatorics of AR quivers]{Nakajima quiver varieties, affine crystals and combinatorics of Auslander-Reiten quivers}
\author{Deniz Kus}
\address{Faculty of Mathematics Ruhr-University Bochum}
\email{deniz.kus@rub.de}
\author{Bea Schumann}
\address{Mathematical Institute, University of Cologne}
\email{bschuman@math.uni-koeln.de}

\begin{abstract}
We obtain an explicit crystal isomorphism between two realizations of crystal bases of finite dimensional irreducible representations of simple Lie algebras of type $A$ and $D$. The first realization we consider is a geometric construction in terms of irreducible components of certain Nakajima quiver varieties established by Saito and the second is a realization in terms of isomorphism classes of quiver representations obtained by Reineke. We give a homological description of the irreducible components of Lusztig's quiver varieties which correspond to the crystal of a finite dimensional representation and describe the promotion operator in type A to obtain a geometric realization of Kirillov-Reshetikhin crystals.
\end{abstract}

\maketitle

\section{Introduction}
Let $\mathfrak{g}$ be a Kac-Moody algebra with symmetric Cartan matrix. A groundbreaking result by Lusztig was the construction of the canonical basis of the negative part $\mathbf{U}_q^-$ of the quantized enveloping algebra of $\mathfrak{g}$ (see \cite{L90,Lu91}). The canonical basis has remarkable properties and yields a basis for any irreducible highest weight $\mathfrak{g}$-module $V(\lambda)$ of highest weight $\lambda$. The main tool of Lusztig's work is given by a certain class of quiver varieties (for a precise definition see Section~\ref{31}) known in the literture as Lusztig's quiver varieties. 

Motivated by Lusztig's work Nakajima introduced another class of quiver
varieties associated to irreducible highest weight $\mathfrak{g}$-modules (see \cite{Na98a} for details). Here a geometric description of the action on $V(\lambda)$ is obtained and certain Lagrangian subvarieties of Nakajima's quiver varieties are defined whose irreducible components yield a geometric basis of $V(\lambda)$.

Using ideas from \cite{L90} an alternative construction of Lusztig's canonical bases of $\mathbf{U}_q^-$ and $V(\lambda)$ was given by Kashiwara in \cite{Ka91}. These bases have well-behaved combinatorial analogues, called the crystal basis $B(\infty)$ of $\mathbf{U}_q^-$ and $B(\lambda)$ of $V(\lambda)$, respectively. The crystal $B(\infty)$ has a geometric realization in terms of irreducible components of Lusztig's quiver varieties established in \cite{KS97}; we denote this realization by $B^g(\infty)$. Moreover, Saito has shown in \cite{Sai} that $B(\lambda)$ also admits a geometric realization in terms of irreducible components of Lagrangian subvarieties of Nakajima's quiver varieties. We denote the geometric realization of Saito in the rest of the paper by $B^g(\lambda)$ and view the crystal graph of $B^g(\lambda)$ as a full subgraph of $B^g(\infty)$ by identifying the irreducible components in $B^g(\lambda)$ with the subset of irreducible components of Lusztig's quiver varieties satisfying a certain stability condition (see Section~\ref{geometricbinfty}). 

The motivation of this paper is to give a homological interpretation of the actions of the crystal operators in $B^g(\lambda)$ using the combinatorics of Auslander-Reiten quivers of a fixed Dynkin quiver $\Qu$ of the same type as $\mathfrak{g}$. This is achieved by constructing an explicit crystal isomorphism to a realization of $B(\lambda)$ in terms of isomorphism classes of $\Qu$-modules introduced by Reineke in \cite{R97}; we denote this realization by $B^{h}(\lambda)$. We first give in Theorem~\ref{irrcomst} a homological description of the stable irreducible components  and then make use of the homological description of $B^g(\infty)$ developed by the second author in \cite{Schu17} and show in Theorem~\ref{crystalstructurelambda} that the embedding of $B^g(\lambda)$ into $B^g(\infty)$ is compatible with this description when $\mathfrak{g}$ is of finite type $A$ or $D$.

In the last part of the paper we consider the standard orientation of the type $A$ quiver and extend the classical crystal structure on the set of irreducible components to an affine crystal structure isomorphic to the Kirillov-Reshetikhin crystal (KR crystal for short). The main idea is to define the promotion operator $\mathrm{pr}$ (see Definition~\ref{defpr}) which is the analogue of the cyclic Dynkin diagram automorphism $i\mapsto i+1\mod (n+1)$ on the level of crystals. This gives, 
together with the homological description in Theorem~\ref{crystalstructurelambda}, a geometric realization of KR crystals (see Corollary~\ref{geoKR}); for combinatorial descriptions of KR crystals in type $A$ we refer the reader to \cite{S02,K12}). 
It will be interesting to discuss the promotion operator for various orientations of the quiver. The connection to Young tableaux (see Section~\ref{section53}) is not known in that case but the description is still possible with several technical difficulties. This construction will be part of forthcoming work.

This paper is organized as follows. In Section~\ref{sec1} we present background material. We recall facts on representations of quivers and the construction of Auslander-Reiten quivers of Dynkin quivers as well as facts on quantum groups and crystal bases. In Section~\ref{geometricbinfty} we recall the geometric construction of crystal bases in terms of irreducible components of quiver varieties due to Kashiwara-Saito \cite{KS97} and Saito \cite{Sai}. In Section~\ref{section41} we develop combinatorics on the geometric constructions using Auslander-Reiten quivers. For this we recall results by Reineke \cite{R97}, prove a criterion for an irreducible component of Lusztig's quiver variety to contain a stable point (Theorem~\ref{irrcomst}) and construct an explicit crystal isomorphism from $B^{g}(\lambda)$ to $B^{h}(\lambda)$ (Theorem~\ref{crystalstructurelambda}). In the final Section~\ref{KRcryssec} we apply the results of Section~\ref{section41} to give a geometric realization of KR-crystals in type $A$. 
 
\section{Background and Notation}\label{sec1}
\subsection{}\label{subsecARquiver} Let $\Qu$ be a Dynkin quiver of type $A$ or $D$ with path algebra $\mathbb{C}\Qu$. Let $I$ be the vertex set of $\Qu$ and $\Qu_1$ the arrow set. For each arrow $h$ of $\Qu$ pointing from the vertex $i$ to the vertex $j$ we write $\out(h)=i$ and $\In({h})=j$. 
Let $\mathfrak{g}$ be the simple Lie algebra associated to the underlying diagram of $\Qu$ over $\mathbb{C}$ with simple system $\{\alpha_i : i\in I\}$, simple coroots $\{h_i:i\in I\}$, Cartan matrix $C=(c_{i,j})_{i,j\in I}$, weight lattice $P$ and dominant integral weights $P^+$. By Gabriel's theorem the isomorphism classes of finite dimensional indecomposable representations of $\Qu$ over $\mathbb{C}$ are in bijection to the set $\Phi^+$ of positive roots of $\mathfrak{g}$. For a positive root $\alpha\in \Phi^+$, we denote by $M(\alpha)$ a representative of the isomorphism class of indecomposable $\mathbb{C}\Qu$-modules associated to this root.

We denote by $\mathbb{C}\Qu-\modd$ the abelian category of finite dimensional left $\mathbb{C}\Qu$-modules. On $\mathbb{C}\Qu-\modd$ we have a non-degenerate bilinear form called the \emph{Euler form} given by:
$$\left<M,N\right>_R:=\dim\Hom_{\mathbb{C}\Qu}(M,N)-\dim\Ext^1_{\mathbb{C}\Qu}(M,N).$$
This form is known to depend only on the dimension vectors $\underline{\dim} M$ and $\underline{\dim} N$ and is equal to
$$\sum_{j\in I}\dim M_j \dim N_j - \sum_{h\in \Qu_1}\dim M_{\out(h)}\dim N_{\In(h)}.$$
The symmetrization of the Euler form
\begin{align*} \left(M,N\right)_R& :=\left<M,N\right>_{R}+\left<N,M\right>_{R} 
\end{align*}
is determined by the Cartan matrix of $\mathfrak{g}$; we have $\left(M,N\right)_R=\mbox{}^{t}(\underline{\dim} M) C (\underline{\dim} N)$ (see \cite[Lemma 3.6.11]{HA6} for details). 

\subsection{} Here we recall briefly the construction of the Auslander-Reiten quiver $\Gamma_{\Qu}$ of $\Qu$ from \cite[Section 6.5]{Ga80}. The vertices of this quiver are given by the isomorphism classes $M$ of indecomposable representations of $\Qu$ while there is an arrow $M\rightarrow N$ if and only if there is an irreducible morphisms $M\rightarrow N$ in $\mathbb{C}\Qu-\modd$ (non-isomorphisms that cannot be written as a composition of two non-isomorphisms). Let $\Qu^*$ be the quiver with the same set of vertices and reversed arrows. As an intermediate step we construct the infinite quiver $\mathbb{Z}\Qu^*$ which has $\mathbb{Z}\times I$ as set of vertices and for each arrow $h: i\rightarrow j$ in $\Qu$ we draw an arrow from $(r,i)$ to $(r+1,j)$ and $(r,j)$ to $(r,i)$ for all $r\in \mathbb{Z}$.
\begin{ex} Consider the Dynkin quiver of type $A_3$ with orientation
$$2\rightarrow 1\leftarrow 3.$$ 
Then $\mathbb{Z}\Qu^*$ is given as follows:

$\tiny
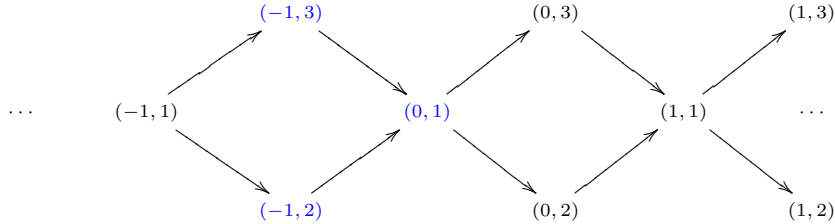
{
\xymatrix{
& & &{\color{blue}(-1,3)} \ar@{->}[rd]& &  (0,3) \ar@{->}[rd]& & (1,3)\\
&\cdots & (-1,1) \ar@{->}[ru] \ar@{->}[rd]& & {\color{blue} (0,1)}\ar@{->}[rd]\ar@{->}[ru] & & (1,1)\ar@{->}[ru] \ar@{->}[rd] &\cdots  & \\
& & & {\color{blue}(-1,2)} \ar@{->}[ru]  & & (0,2) \ar@{->}[ru] &  &(1,2) \\
}}$
\end{ex}
A slice of $\mathbb{Z}\Qu^*$ is a connected full subquiver which contains for each $i\in I$ a unique vertex of the form $(r,i)$, $r\in \mathbb{Z}$. There is a unique slice $S_{\Qu}$ which contains $(0,1)$ and is isomorphic to $\Qu$ (in the above example highlighted in blue). The Nakayama permutation $\nu:\mathbb{Z}\times I\rightarrow \mathbb{Z}\times I$ gives a bijective correspondence between the indecomposable projectives of $\mathbb{C}\Qu$ and the indecomposable injectives of $\mathbb{C}\Qu$. For example, in type $A_n$ we have 
$$\nu(r,i)=(r+i-1, n+1-i).$$ Now $\Gamma_{\Qu}$ can be identified with the full subquiver of $\mathbb{Z}\Qu^{*}$ formed by the vertices lying between $S_{\Qu}$ and the image of this slice under $\nu$ (see \cite[Proposition 6.5]{Ga80}).
Furthermore $\mathbb{Z}\Qu^{*}$ has a translation structure given by the  Auslander--Reiten translation $\tau$ given by $\tau(p,q)=(p-1,q)$. This function gives rise to a bijection between the isomorphism classes of indecomposable non--projectives and the isomorphism classes of indecomposable non--injectives when restricted to $\Gamma_{\Qu}$. We can describe $\tau$ as a function on the dimension vectors. For $i\in I$ define $r_i:\mathbb{Z}_{\ge 0}^{|I|}\rightarrow \mathbb{Z}_{\ge 0}^{|I|} $ via
$$r_i(v)=v-(v,e_i)_R e_i.$$
Here we denote the dimension vector of the simple $\mathbb{C}\Qu$-module $S(i)$ by $e_i$. We fix a labeling $i_1,i_2,\dots,i_n$ of $I$ adapted to $\Qu$, that is $i_1$ is a sink of $\Qu$ and $i_2$ is a sink of $\sigma_1\Qu$ and so on ($\sigma_1$ revereses the direction of all arrows at vertex $1$). For an indecomposable non-projective $\mathbb{C}\Qu$-module $M$, the indecomposable non-injective $\mathbb{C}\Qu$-module $\tau M$ has dimension vector 
$$r_{i_n}r_{i_{n-1}}\cdots r_{i_1}(\underline{\dim} M).$$ If we consider the quiver $1\leftarrow 3\rightarrow 2$ and denote a vertex of  $\Gamma_{\Qu}$ by the dimension vector of its isomorphism class we get

$\tiny{
\xymatrix{
& & && & 011 \ar@{->}[rd]& & 100 \ar@{-->}[ll]^{\tau}\\
&&  & & 010\ar@{->}[rd]\ar@{->}[ru] & & 111\ar@{->}[ru]\ar@{-->}[ll]^{\tau} \ar@{->}[rd] & & \\
& & & & & 110 \ar@{->}[ru] &  &001 \ar@{-->}[ll]^{\tau}\\
}}$
\subsection{}\label{section23}Let $(\mathbb{C}\Qu)^{\op}$ be the opposite algebra of $\mathbb{C}\Qu$. We have a functor $\mathcal{D} :\mathbb{C}\Qu-\modd \rightarrow (\mathbb{C}\Qu)^{\op}-\modd$, called \emph{standard duality functor}, between the category of $\mathbb{C}\Qu$-modules and the category of $(\mathbb{C}\Qu)^{\op}$-modules. For $M\in \mathbb{C}\Qu-\modd$ we have $\mathcal{D}(M):=\Hom_{\mathbb{C}}(M,\mathbb{C})$ with $(\mathbb{C}\Qu)^{\op}$-module structure defined by
$$(a\varphi)(m)=\varphi(am),\ \ m\in M,\ a\in (\mathbb{C}\Qu)^{\op},\ \varphi \in \mathcal{D}(M).$$ Furthermore, for $M,N \in \mathbb{C}\Qu-\modd$, $f\in \Hom_{\mathbb{C}\Qu}(M,N)$, we have $\mathcal{D}(f):\mathcal{D}(N)\rightarrow \mathcal{D}(M), \ \phi \mapsto \phi \circ f$. From the definitions it is straightforward to see that $\mathcal{D}(\tau M)=\tau^{-1} \mathcal{D}(M)$. Also we can identify representations of $(\mathbb{C}\Qu)^{\op}-\modd$ with representations of $\mathbb{C}{\Qu^*}-\modd$. Thus the Auslander-Reiten quiver of $\Qu^*$ can be obtained by reversing each arrow in the Auslander-Reiten quiver of $\Qu$ and interchanging the roles of $\tau$ and $\tau^{-1}$.

\subsection{}
Let $\mathbf{U}_q(\mathfrak{g})$ be the $\mathbb{Q}(q)$-algebra with generators $E_i,F_i,K_i^{\pm 1}$, $i\in I$ and the following relations for $j\in I\setminus \{i\}$
$$K_iK_i^{-1}=K_i^{-1}K_i=1, \quad K_iK_j=K_jK_i, \quad K_iE_iK_i^{-1}=q^{2}E_i$$
$$K_iF_iK_i^{-1}=q^{-2}F_i, \quad E_iF_j-F_jE_i=0, \quad E_iF_i-F_iE_i=\frac{K_i-K_i^{-1}}{q-q^{-1}}$$
\begin{align*} \text{If }c_{i,j}=-1: \ &  E_i^{2}E_j+E_jE_i^2=(q+q^{-1})E_iE_jE_i, \quad F_i^{2}F_j+F_jF_i^2=(q+q^{-1})F_iF_jF_i, \\
& K_iE_jK_i^{-1}=q^{-1}E_j, \quad K_iF_jK_i^{-1}=qF_j.
\end{align*}
\begin{align*} \text{If }c_{i,j}=0: \ & E_iE_j=E_jE_i, \quad F_iF_j=F_jF_i, \quad K_iE_jK_i^{-1}=E_j, \quad K_iF_jK_i^{-1}=F_j.
\end{align*}

For $m\in \N$, let $[m]_{q}:=q^{m-1}+q^{m-3}+\cdots +q^{-m+1}$ and define for $x\in \mathbf{U}_q(\mathfrak{g})$ the divided power
\begin{equation}\label{dividedpower}
x^{(m)}:=\frac{x^m}{[m]_{q}!}.
\end{equation}

For $\lambda\in P^+$ we denote by $V(\lambda)$ the irreducible $\mathbf{U}_q(\mathfrak{g})$-module of highest weight $\lambda$ and let $\mathbf{U}_q^-\subseteq \mathbf{U}_q(\mathfrak{g})$ be the subalgebra generated by $F_i$, $i\in I$.

\begin{defi} \label{abstrcryst} An abstract $\mathfrak{g}$-crystal $B$ is a set endowed with maps
\begin{align*}
\wt: B &\rightarrow P, &
\varepsilon_i : B &\rightarrow \mathbb{Z}\sqcup \{-\infty\}, \quad  \varphi_i:B \rightarrow  \mathbb{Z}\sqcup \{-\infty\}, \\
\tilde{e}_i: B &\rightarrow B \sqcup \{0\}, \quad &\tilde{f}_i:B&\rightarrow B \sqcup \{0\} \quad \text{ for }i\in I.
\end{align*}
satisfying the following axioms for $i\in I$ and $b, b'\in B$
\begin{itemize}
\item $\varphi_i(b)=\varepsilon_i(b)+\wt(b)(h_i)$,
\item if $b\in B$ satisfies $\tilde{e}_i b \ne 0$ then $$\wt(\tilde{e}_i b)=\wt(b) + \alpha_i, \quad \varphi_i(\tilde{e}_i b)=\varphi_i(b)+1,\quad \varepsilon_i(\tilde{e}_i b)=\varepsilon_i(b)-1,$$
\item if $b\in B$ satisfies $\tilde{f}_i b \ne 0$ then $$\wt(\tilde{f}_i b)=\wt(b) - \alpha_i, \quad \varphi_i(\tilde{f}_a b)=\varphi_i(b)-1, \quad \varepsilon_i(\tilde{f}_i b)=\varepsilon_i(b)+1,$$
\item $\tilde{e}_i b =b'$ if and only if $\tilde{f}_i b' =b$,
\item if $\varepsilon_i (b) = - \infty$, then $\tilde{e}_i b=\tilde{f}_i b =0$.
\end{itemize}
\end{defi}

Let $B_1$ and $B_2$ be abstract $\mathfrak{g}$-crystals. A map $\psi:B_1 \sqcup \{0\} \rightarrow B_2\sqcup \{0\}$ satisfying $\psi(0)=0$ is called a \emph{morphism of crystals} if for $b\in B_1$, $\psi(b)\in B_2$ and $i\in I$ we have
$$
\wt(\psi(b)) =\wt(b), \quad \varepsilon_i(\psi(b))=\varepsilon_i(b), \quad \varphi_i(\psi(b))=\varphi_i(b),
$$
$$\psi(\tilde{e}_ib)=\tilde{e}_i\psi(b),\ \text{if  $\tilde{e}_ib\neq 0$}\ \ \ \ \psi(\tilde{f}_ib)=\tilde{f}_i\psi(b),\ \text{if  $\tilde{f}_ib\neq 0$}.
$$
A morphism of crystals which commutes with all $\tilde{e}_i,\tilde{f}_i$ is called \textit{strict morphism of crystals}. An injective (strict) morphism is called a \emph{(strict) embedding of crystals} and a bijective strict morphism is called an \textit{isomorphism of crystals}.

Let $B_1$ and $B_2$ be abstract $\mathfrak{g}$-crystals. The set 
$$B_1\otimes B_2:=\{b_1 \otimes b_2 \mid b_1 \in B_2, \ b_2\in B_2\}$$ 
is called the tensor product of $B_1$ and $B_2$ and admits a $\mathfrak{g}$-crystal structure where $\wt(b_1 \otimes b_2)=\wt(b_1)+\wt(b_2)$ and 
\begin{align*}
\varepsilon_i(b_1\otimes b_2)&= \max\{\varepsilon_i(b_1),\varepsilon_i(b_2)-\wt(b_1)(h_i)\}, \\
\varphi_i(b_1\otimes b_2)&= \max\{\varphi_i(b_2),\varphi_i(b_1)+\wt(b_2)(h_i)\}, \\
\tilde{e}_i(b_1\otimes b_2) & = \begin{cases} \tilde{e}_i b_1 \otimes b_2 & \text{ if } \varphi_i(b_1) \ge \varepsilon_i(b_2) \\ b_1 \otimes \tilde{e}_i b_2 & \text{ else,} \end{cases} \\
\tilde{f}_i(b_1\otimes b_2) & = \begin{cases} \tilde{f}_i b_1 \otimes b_2 & \text{ if } \varphi_i(b_1) > \varepsilon_i(b_2) \\ b_1 \otimes \tilde{f}_i b_2 & \text{ else.} \end{cases}
\end{align*}

\subsection{}\label{sec:crysrep}

We recall the definition of the \emph{crystal bases} $B(\infty)$ and $B(\lambda)$ of $\mathbf{U}_q^{-}$ and $V(\lambda)$, respectively, following \cite[Sections 2 and 3]{Ka91}. We fix $i\in I$ in the rest of the discussion. For $P\in \mathbf{U}_q^{-}$ there exists unique $Q,R \in \mathbf{U}_q^{-}$ such that 
$$E_iP-PE_i=\frac{K_iQ-K_i^{-1}R}{q-q^{-1}}.$$ 
The endomorphism $E_i': \mathbf{U}_q^{-}\rightarrow \mathbf{U}_q^{-}$ given by $E'_i(P)=R$ induces a vector spaces decomposition
$$\mathbf{U}_q^{-}=\displaystyle\bigoplus_{m\ge 0} F_i^{(m)}\mathrm{Ker}(E'_i).$$
We define the \emph{Kashiwara operators } $\tilde{e}_i$, $\tilde{f}_i$ on $\mathbf{U}_q^-$ by the following rule
\begin{equation*}
\tilde{f}_i(F_i^{(m)}u)= F_i^{(m+1)}u, \quad \tilde{e}_i(F_i^{(m)}u)= F_i^{(m-1)}u,\ \ \ u\in \mathrm{Ker}(E_i').
\end{equation*}
Let $\mathcal{A}$ be the subring of $\mathbb{Q}(q)$ consisting of rational functions $f(q)$ without a pole at $q=0$. Let $\mathcal{L}(\infty)$ be the $\mathcal{A}$-submodule generated by all elements of the form
\begin{equation}\label{seqKas}
\tilde{f}_{i_1}\tilde{f}_{i_2}\cdots\tilde{f}_{i_{\ell}}(1)
\end{equation}
and let $B(\infty)\subseteq \mathcal{L}(\infty)/q\mathcal{L}(\infty)$ be the subsets of all residues of \eqref{seqKas}. For $b\in B(\infty)$ we let $\wt(b)$ be the weight of the element and $\varepsilon_i(b)=\max\{\tilde{e}_i^k(b) \ne 0 \mid k\in \mathbb{N}\}.$ This endows $B(\infty)$ with the structure of an abstract crystal (see Definition \ref{abstrcryst}).
The $\mathbb{Q}(q)$-antiautomorphism $*:\mathbf{U}_q^- \rightarrow \mathbf{U}_q^-$ given by
$$E_i^{*}=E_i,\ F_i^{*}=F_i,\ K_i^{*}=K_{i}^{-1}$$
has the properties $\mathcal{L}(\infty)^{*}=\mathcal{L}(\infty)$, $B(\infty)^{*}=B(\infty)$ and $*$ preserves the weight function (see \cite[Theorem 8.3]{K95} for details). We can endow $B(\infty)$ with a second structure of an abstract crystal denoted by  $B(\infty)^*$ with Kashiwara operators (also called the $*$-twisted maps)
$$\tilde{f}_i^*(x)=(\tilde{f}_ix^*)^{*},\ \ \tilde{e}_i^*(x)=(\tilde{e}_ix^*)^{*},\ \ \varepsilon_i^*(x)=\varepsilon_i(x^*)$$
By construction $*$ induces a crystal isomorphism between $B(\infty)$ and $B(\infty)^*$. For $\lambda\in P^+$ let $\pi_{\lambda}:\mathbf{U}_q^- \rightarrow V(\lambda)$ be the natural $\mathbf{U}_q^-$-homomorphism sending $1$ to $v_{\lambda}$ and denote by $(\mathcal{L}(\lambda),B(\lambda))$ the crystal base of $V(\lambda)$. Then we have (see \cite[Theorem 5]{Ka91}) $\pi_{\lambda}(\mathcal{L}(\infty))=\mathcal{L}(\lambda)$ and we obtain an induced surjective homomorphism 
$$\overline{\pi}_{\lambda}: \mathcal{L}(\infty)/q\mathcal{L}(\infty)\rightarrow \mathcal{L}(\lambda)/q\mathcal{L}(\lambda)$$
with the following properties 
\begin{itemize}
\item $B(\lambda)$ is isomorphic to $\{\overline{\pi}_{\lambda}(b): b\in B(\infty)\}\backslash\{0\}$ by the map $\overline{\pi}_{\lambda}$.
\item $\tilde{f}_i\circ \overline{\pi}_{\lambda}=\overline{\pi}_{\lambda} \circ \tilde{f}_i$ for all $i\in I$,
\item If $b\in B(\infty)$ is such that $\overline{\pi}_{\lambda}(b)\ne 0$, then we have $\tilde{e}_i\overline{\pi}_{\lambda}(b)=\overline{\pi}_{\lambda}(\tilde{e}_i b)$ for all $i\in I$.
\end{itemize}
We denote by $T_{\lambda}=\{r_{\lambda}\}$ the abstract $\mathfrak{g}$-crystal consisting of one element and 
$$\wt(t_{\lambda})=\lambda,\ \varepsilon_i(t_{\lambda})=\varphi_i(t_{\lambda})=-\infty,\ \tilde{e}_i t_{\lambda}=\tilde{f}_i t_{\lambda}=0,\ \  \forall i\in I.$$ By the above properties we have an embedding of crystals
$$B(\lambda)\rightarrow B(\infty)\otimes T_{\lambda},\ \ \overline{\pi}_{\lambda}(b)\mapsto b\otimes t_{\lambda}$$ which commutes with the $\tilde{e}_i$'s (but not necessarily with the $\tilde{f}_i$'s) whose image is given by (see \cite[Proposition 8.2]{K95})
\begin{equation}\label{blambdainbinfty}
\{b\otimes t_{\lambda} \in B(\infty) \otimes T_{\lambda} : \varepsilon^*_i(b) \le \lambda(h_i) \ \forall i \in I\}.
\end{equation}
We summarize the above discussion in the following theorem.
\begin{thm}[{\cite[Proposition 8.2]{K95}}]\label{kashiwara} Let $\lambda\in P^{+}$. Then the crystal graph $B(\lambda)$ of the irreducible highest weight module $V(\lambda)$ can be realized as the full subgraph of $B(\infty)$ consisting of all vertices $b\in B(\infty)$ such that $\varepsilon^*_i(b)\le \lambda(h_i)$ for all $i\in I$. For $b\in B(\lambda)$ the Kashiwara operators $\tilde{f}_i^{\lambda},\tilde{e}_i^{\lambda}$ on $B(\lambda)$ are given by
\begin{equation*}
\tilde{e}_i^{\lambda}b=\tilde{e}_i b, \qquad \tilde{f}^{\lambda}_i b =\begin{cases} 0,& \text{ if } \tilde{f}_ib\notin B(\lambda) \\ \tilde{f}_i b,& \text{ if }\tilde{f}_ib\in B(\lambda). \end{cases} 
\end{equation*}
\qed
\end{thm}
\section{Geometric construction of crystal bases}\label{geometricbinfty}

In this section we review a geometric construction of the crystals $B(\infty)$ and $B(\lambda)$ for $\lambda\in P^+$ in terms of irreducible components of quiver varieties. 

\subsection{}\label{31}We denote by $\overline{\Qu}=(I,H)$ the associated \emph{double quiver} of $\Qu$, where for each $h\in \Qu_1$, $H$ contains two arrows with the same endpoints, one in each direction. For an arrow $h\in H$, we denote by $\overline{h}$ the arrow with $\out(h)=\In(\overline{h})$ and $\In(h)=\out(\overline{h})$. In this notation, the double quiver $\overline{\Qu}=(I,H)$ has as set of arrows $H=\Qu_1 \sqcup \overline{\Qu_1}$.

\begin{ex} For $\Qu= 1 \xleftarrow{h_1} 2 \xleftarrow{h_2} 3$, the double quiver $\overline{\Qu}$ looks as follows.
\begin{equation*}
\overline{\Qu}=\xymatrix{
1\ar@<0.5ex>[r]^{\overline{h}_1} & 2\ar@<0.5ex>[l]^{h_1} \ar@<0.5ex>[r]^{\overline{h}_2} & 3.\ar@<0.5ex>[l]^{h_2}
}\
\end{equation*}
\end{ex}
 We define the function
\begin{align*}
\epsilon: H & \rightarrow \{\pm 1\} \\
h&\mapsto \left\{\begin{array}{cl} 1, & \mbox{if }h\in \Qu_1 \\ -1, & \mbox{if }h \notin \Qu_1 .\end{array}\right.
\end{align*}
The \emph{preprojective algebra $\Pi(\Qu)$} is the quotient of the path algebra of the double quiver $\overline{\Qu}$ by the ideal generated by
$$\sum_{h\in H, \In(h)=i}\epsilon(h)h\bar{h},\ \ \ i\in I.$$
For a fixed finite--dimensional $I$--graded vector spaces $V=\bigoplus_{i\in I} V_i$ over $\mathbb{C}$, we define \emph{Lusztig's quiver variety} $\Lambda_V$ to be the variety of
representations of $\Pi(\Qu)$ with underlying vector space $V$, i.e.
$$\Lambda_V:=\left\{(x_h)_{h\in H}\in \displaystyle \bigoplus_{h\in H}{\Hom_{\mathbb{C}}(V_{\out(h)},V_{\In(h)})}: \sum_{h\in H, \In(h)=i}\epsilon(h)x_hx_{\bar{h}}=0 \text{ for all }i \in I
\right\}.$$
Let $\Repp_V(\Qu)$ be the variety of representations of $\Qu$ with underlying vector space $V$, that is
$$\Repp_V(\Qu)=\displaystyle \bigoplus_{h\in \Qu_1}{\Hom(V_{\out(h)},V_{\In(h)})},$$
which is clearly a closed subvariety of the affine variety $\Lambda_V$. From now on we constantly identify the points of $\Lambda_V$ (resp. $\Repp_{V}(\Qu)$) with the corresponding modules over $\Pi(\Qu)$ (resp. $\mathbb{C}\Qu$) and write expressions like $M\in
\Lambda_V$ for $M=(V,x)\in \Pi(\Qu)-\modd$. We have an action of the group $G_v=\prod_i{GL(V_i)}$ on $\Lambda_V$ and $\Repp_{V}(\Qu)$ by base change, that is for $M=(V,x) \in \Lambda_V$, $g.M=\widetilde{M}$, where
$\widetilde{M}=(V,\widetilde{x})\in \Pi(\Qu)-\modd$ with
$$\widetilde{x}_h:=g_{\In(h)}x_hg^{-1}_{\out(h)},\ \ h\in H$$
and analogously for $M\in \Repp_{V}(\Qu)$. The orbits of this action on $\Lambda_V$ (resp. $\Repp_{V}(\Qu)$) are exactly the isomorphism classes of representations of $\Pi(\Qu)-\modd$ (resp. $\mathbb{C}\Qu-\modd$) with fixed dimension
vector $v:=\underline{\dim}(V)$. For $M\in  \Repp_{V}(\Qu)$, we denote the corresponding orbit by $\mathcal{O}_M$ and let $\mathfrak{gl}_v=\bigoplus_{i\in I} \mathfrak{gl}_{v_i}$ the Lie algebra of $G_v$.

\begin{rem} The definition of preprojective algebras is motivated by symplectic geometry, namely Lusztig's quiver variety can be viewed as the zero fibre of the moment map for the action of $G_v$ on $\Repp_V(\overline{\Qu})$.
The additional nilpotency condition on the elements of $\Lambda_V$ is omitted, since we restrict ourselves to preprojective algebras of Dynkin quivers and this condition is automatically satisfied (see \cite[Proposition 14.2(a)]{Lu91}).
\end{rem}
Note that, up to isomorphism, $\Lambda_V$ depends only on the graded dimension of $V$. Therefore we also denote $\Lambda_V$ by $\Lambda(v)$,  regarding the graded dimension of the vector spaces as part of the datum of the representations of $\Pi(\Qu)$. The next lemma describes the irreducible components of the variety $\Lambda(v)$.
\begin{lem}\label{irrcomplusztig}
An element $(x=(x_h),\overline{x}=(x_{\overline{h}}))_{h\in Q_1}\in \Repp_{V}(\Qu)\oplus \Repp_{V}(\Qu^{*})$ lies in the quiver variety $\Lambda(v)$ if and only if 
$$\mathrm{tr}([a,x]\overline{x}):=\sum_{h\in Q_1}\mathrm{tr}((a_{\text{in}(h)}x_h-x_ha_{\text{out}(h)})x_{\overline{h}})=0,\  \text{for all $a\in \mathfrak{gl}_v$}.$$
The irreducible components are exactly the closures of conormal bundles to $G_v$-orbits in $\Repp_{V}(\Qu)$:
$$\Lambda(v)=\bigcup_{\mathcal{O}_M\in  \Repp_{V}(\Qu)/G_v}\ \overline{\bigcup_{x\in\mathcal{O}_M}(\{x\}\times X_x)}.$$
where $X_x=\left\{\overline{x}\in \Repp_{V}(\Qu^{*}):  \mathrm{tr}([a,x]\overline{x})=0\ \forall a\in \mathfrak{gl}_v\right\}.$
\proof
For the first claim see \cite[Lemma 5.6]{Kjr16} and hence $\Lambda(v)$ is in fact the union of these conormal bundles. Since $G_v$ is irreducible, we also have that the closures of conormal bundles are irreducible subvarieties of $\Lambda(v)$. The fact that they are as well the irreducible components follows from Gabriel's theorem,  since the orbit space $\Repp_{V}(\Qu)/G_v$ is finite.
\endproof

\end{lem}
We denote the irreducible components of $\Lambda(v)$ by $\Ir\Lambda(v)$ and the 
closure of the conormal bundle corresponding to the orbit $\mathcal{O}_M$ by $X_{M}$.
 
\subsection{}\label{gc1}Now, we recall the geometric construction of Kashiwara operators on the set of irreducible components $\Ir\Lambda(v)$ from \cite{KS97}. For $i\in I$ and $M\in\Lambda(v)$, define $\varepsilon_i(M)$ to be the dimension of the $S(i)$-isotypic component of the head of $M$. For $M=(V,x)\in\Pi(\Qu)-\modd$, that is
\begin{equation}\label{eq:epsil}
{\varepsilon_{i}(M)}=\dim\Coker\left(\displaystyle\bigoplus_{h:\In(h)=i}V_{\out(h)}\xrightarrow{x_h}{V_i}\right).
\end{equation}
For $c\in \mathbb{Z}_{\ge 0}$, we further introduce the subsets
\begin{equation*}
\Lambda(v)_{i,c}:=\{M\in \Lambda(v)  \mid \varepsilon_i(M)=c\}.
\end{equation*}
Let $e_i\in \mathbb{Z}_{\ge 0}^{|I|}$ be as usual the $i$-th unit vecor and fix $c\in \mathbb{Z}_{\ge 0}$ such that $v_i-c\ge 0$. We define $$\Lambda(v,c,i):=\{(M,N,\varphi) : M\in \Lambda(v)_{i,c},\ N\in \Lambda(v-ce_i)_{i,0},\ \varphi\in \Hom_{\Pi(\Qu)}(N,M) \text{
injective}\}.$$
Considering the diagram
\begin{equation}\label{diaggeom}
\Lambda(v-ce_i)_{i,0} \xleftarrow{p_1} \Lambda(v,c,i)\xrightarrow{p_2}\Lambda(v)_{i,c},
\end{equation}
where $p_1(M,N,\varphi)=N$ and $p_2(M,N,\varphi)=M$. It is shown in \cite[Lemma 5.2.3]{KS97} that the map $p_2$ is a principal $G_v$-bundle and the map $p_1$ is a smooth map whose fibres are connected varieties. Standard algebraic geometry arguments then show (for $\Lambda(v)_{i,c}\ne\emptyset$) that there is a one--to--one correspondence between the set of irreducible
components of $\Lambda(v-ce_i)_{i,0}$ and the set of irreducible components of $\Lambda(v)_{i,c}$, i.e. 
\begin{equation}\label{eq:irrbijectionn1}
\Ir\Lambda(v-ce_i)_{i,0}\cong\Ir\Lambda(v)_{i,c}.
\end{equation}
Let $X\in \Ir\Lambda(v)$ and define for $i\in I$ the integer $\varepsilon_i(X):=\min_{M\in X}\varepsilon_i(M)$.  The function $\varepsilon_i$ given in (\ref{eq:epsil}) is upper semi-continuous, i.e.
$$\{M\in X: \varepsilon_i(M)\geq \varepsilon_i(X)+1\}$$
is a closed subset. So there is an open dense subset of $X$ such that
$\varepsilon_i$ is constant (namely the value of $\varepsilon_i$ on this subset is 
$\varepsilon_i(X)$).
Let
$$\mathrm{I}\Lambda(v)_{i,c}:=\{X\in \Ir\Lambda(v)\mid\varepsilon_i(X)=c\}.$$
So if $X\in  \mathrm{I}\Lambda(v)_{i,c}$, we obtain from the prior considerations that there is an open dense subset $U_X$ of $X$ such that $U_X\subseteq \Lambda(v)_{i,c}$.  Since $\Lambda(v)$ and $\Lambda(v)_{i,c}$ have pure dimension $\frac{1}{2}\dim \Repp_V(\Qu)$ (see \cite[Theorem 12.3]{Lu91}), we get a bijection 
\begin{equation}\label{eq:irrbijectionn}
 \mathrm{I}\Lambda(v)_{i,c}\rightarrow \mathrm{I}\Lambda(v-ce_i)_{i,0},\ \ X\mapsto \overline{p_1p_2^{-1}(U_X)}.
\end{equation}
Suppose that $\bar{X}\in \mathrm{I}\Lambda(v-ce^i)_{i,0}$ corresponds to the component $X\in \mathrm{I}\Lambda(v)_{i,c}$ und the bijection \ref{eq:irrbijectionn}. We define maps 
\begin{align*}
\tilde{f}_i^c:\mathrm{I}\Lambda(v-ce_i)_{i,0}&\rightarrow \mathrm{I}\Lambda(v)_{i,c},\ \ \bar{X}\mapsto X\\
\tilde{e}_i^{c}:\mathrm{I}\Lambda(v)_{i,c}&\rightarrow \mathrm{I}\Lambda(v-ce^i)_{i,0},\ \ X\mapsto \bar{X} \label{actione}
\end{align*}
The data of these maps yields a crystal structure on $B^g(\infty):=\bigsqcup_{v}\Ir\Lambda(v)$ together with the following maps for $X\in \mathrm{I}\Lambda(v)_{i,c}\subseteq B^g(\infty)$:
$$
\tilde{f}_i(X):=\tilde{f}_i^{c+1}\tilde{e}_i^c(X),\ \ \ \tilde{e}_i(X):=\begin{cases}\tilde{f}_i^{c-1}\tilde{e}_i^c(X),& \text{ if $c>0$}\\
0,& \text{ otherwise}
\end{cases}
$$
$$
\wt(X):=-\displaystyle\sum_{i\in I}v_i\alpha_i \text{ for } X\in \Ir\Lambda(v),\ \ \ \ 
\varphi_i(X):=\varepsilon_i(X)+\wt(X)(h_i).
$$
It is shown in \cite[Theorem 5.3.2]{KS97} that $B^g(\infty)$ is isomorphic to the crystal $B(\infty)$ of $\mathbf{U}_q^-$.
\subsection{} In \cite{Sai}, Saito gave a realization of the crystal $B(\lambda)$ via Nakajima's quiver varieties. In this section we recall the definition of those spaces. To define them we consider a framing on the double quiver $\overline{\Qu}$ by adding an extra vertex $i'$ and an extra arrow $t_i:i\rightarrow i'$ for all $i\in I$.

\begin{ex} The framed double quiver for the Dynkin graph of type $A_3$ is given as follows.
\begin{equation*}
\xymatrix{
1\ar@<0.5ex>[r]^{\overline{h}_1}\ar@<0.5ex>[d]^{t_1} & 2\ar@<0.5ex>[l]^{h_1}\ar@<0.5ex>[d]^{t_2} \ar@<0.5ex>[r]^{\overline{h}_2} & 3\ar@<0.5ex>[l]^{h_2}\ar@<0.5ex>[d]^{t_3} \\
1' & 2' & 3'
\
}
\end{equation*}
\end{ex}
For $v,\lambda \in \mathbb{Z}_{\ge 0}^{|I|}$, we choose $I$-graded vector spaces $V$ and $W$ of graded dimension $v,\lambda$, respectively and define
$$\Lambda(v,\lambda):=\Lambda(v) \times \displaystyle\bigoplus_{i\in I} \Hom(V_i,W_i).$$
The action of the group $G_v$ can be extended on $\Lambda(v,\lambda)$ via
$$g  (x,t) :=(g_i)_{i\in I}  \left((x_h)_{h\in H}, (t_i)_{i\in I}\right)=\left((g_{\In(h)}x_h g_{\out(h)}^{-1})_{h\in H},(t_i g_i^{-1})_{i\in I}\right).$$
We consider the subset $\Lambda(v,\lambda)^{st}$ of stable points in $\Lambda$ defined as
$$\Lambda(v,\lambda)^{st}:=\{(x,t)\in \Lambda(v,\lambda):\displaystyle\bigcap_{\out(h)=i}\left(\ker x_h \cap \ker t_i\right) = 0\}.$$
\begin{rem} This definition is equivalent to the one given in \cite{Na98a} stating that there is no non-trivial $x$-stable subspace of $V$ contained in the kernel of $t$, see \cite[Lemma 3.4]{FS03}. Moreover, this is a stability condition in the sense of Mumford with respect to the character $\theta=(\theta_i)\in \mathbb{Z}^{|I|}$ of $G_v$ with $\theta_i=-1$ for all $i\in I$ (see \cite[Section 3.2]{Na98a}).
\end{rem}
The subset $\Lambda(v,\lambda)^{st}$ is open in $\Lambda(v,\lambda)$ and we clearly have an induced action of the group $G_v$ on $\Lambda(v,\lambda)^{st}$. We further have the following.
\begin{lem}[{\cite[Lemma 3.1]{Na98a}}]\label{free} The action of $G_v$ on $\Lambda(v,\lambda)^{st}$ is free and $\Lambda(v,\lambda)^{st}$ is a non-singular subvariety of $\Lambda(v,\lambda)$.
\qed
\end{lem}
The \textit{Nakajima quiver variety} is defined to be the geometric quotient of $\Lambda(v,\lambda)^{st}$ by $G_v$
$$\mathfrak{L}(v,\lambda):={\Lambda(v,\lambda)^{st}}/{G_v}.$$ We denote by ${ \Ir\mathfrak{L}(v,\lambda)}$ the set of irreducible components of $\mathfrak{L}(v,\lambda)$. Now using Lemma~\ref{free} we can make the following observations. We have 
\begin{equation}\label{eq:comp}
\Ir\mathfrak{L}(v,\lambda)=\left\{{Y^{\lambda}_{M}}: M \in \mathbb{C}\Qu-\modd\text{ and }Y^{\lambda}_{M}\ne \emptyset\right\}
\end{equation}
where 
$$Y^{\lambda}_{M}:={\left(\left(X_{M}\times \displaystyle\bigoplus_{i\in I}\Hom(V_i,W_i)\right)\cap \Lambda(v,\lambda)^{st}\right)}/{G_v}.$$
Moreover, we have the following identification
\begin{equation}\label{eq:comp1}\Ir\mathfrak{L}(v,\lambda)\cong\{Y\in \Ir\Lambda(v,\lambda): Y\cap \Lambda(v,\lambda)^{st}\ne \emptyset\}.\end{equation}

We conclude that the irreducible components of $\mathfrak{L}(v,\lambda)$ are in one-to-one correspondence to the irreducible components of $\Lambda(v,\lambda)$ that contain a stable point. In \cite{Sai} Saito describes a crystal structure on $\Ir\mathfrak{L}(v,\lambda)$ using similar arguments as in \cite{KS97}.  The key point for our approach is the following theorem.

\begin{thm}[{\cite[Theorem 4.6.4, Lemma 4.6.3]{Sai}}]\label{injectivemap} The map $\mathfrak{i}:\Ir\mathfrak{L}(v,\lambda) \rightarrow B^{g}(\infty),\ Y^{\lambda}_{M}\mapsto X_{M}$ is an embedding of crystals which commutes with the operators $\tilde{e}_i$, $i\in I$. Moreover, $\Ir\mathfrak{L}(v,\lambda)$ is isomorphic to the crystal $B^g(\lambda)$ and hence $B^g(\lambda)$ is the full subgraph of $B^g(\infty)$ with vertices $$\{X_{M}\in B^g(\infty): Y^{\lambda}_{M}\ne \emptyset\}.$$\qed
\end{thm}
A description of the irreducible components \eqref{eq:comp} will be given purely combinatorial in the Auslander-Reiten quiver $\Gamma_{\Qu}$ in Theorem~\ref{irrcomst}.
\section{Crystals via the combinatorics of Auslander-Reiten quivers}
\subsection{}\label{section41}In this subsection we recall some fundamental definitions from \cite{R97} which arise in the context of an explicit crystal structure on certain Lusztig PBW basis. 
We define the posets 
$$\mathcal{P}_i(\Qu):=\{M\in \mathbb{C}\Qu-\modd : M \text{ is indecomposable and }\dim\Hom_{\mathbb{C}\Qu}(M,S(i))\ne 0\}$$
$$\mathcal{P}^{\scriptscriptstyle\vee{}}_i(\Qu):=\{M\in \mathbb{C}\Qu-\modd : M \text{ is indecomposable and }\dim\Hom_{\mathbb{C}\Qu}(S(i),M)\ne 0\}$$ 
together with the relation $\preceq$ given by
$$N\preceq M \iff \Hom_{\mathbb{C}\Qu}(N,M)\ne 0.$$
Furthermore we define the poset  $${\mathcal{S}_i(\Qu)}:=\left\{V=\bigoplus_{j=1}^{k}M_j: \{M_1,\dots, M_k\} \text{ forms an antichain in } \mathcal{P}_i(\Qu)\right\}$$ together with the relation $\unlhd$ given by
$$V \kg V' \iff \dim\Hom_{k\Qu}(B, V') \ne 0 $$
for each indecomposable direct summand $B$ of $V$.
In the same way we can define ${\mathcal{S}^{\vee}_i(\Qu)}$ with the relation $\unlhd^{\vee}$ given by
$$V \kg^{\ch} V' \iff \dim\Hom_{k\Qu}(V, B) \ne 0 $$
for each indecomposable direct summand $B$ of $V'$.
\begin{ex}
Let $\Qu$ be the following quiver
\begin{equation*}
\xymatrix{ & 2 & \\
1 & 3 \ar@{->}[l]\ar@{->}[u] & 4. \ar@{->}[l]}
\end{equation*}
The poset $\mathcal{P}^{\ch}_1(\Qu)$ is the union of all framed modules:

\begin{center}\resizebox{11cm}{!}{$
\xymatrix{
\fbox{{\bf\text{$\begin{matrix}  & 0 &  \\ 1 & 0 & 0\end{matrix}$}}} \ar@{->}[rdd]& & \text{$\begin{matrix}  & 1 &  \\ 0 & 1 & 0\end{matrix}$}\ar@{->}[rdd]\ar@{-->}[ll]^{\tau} & & \fbox{{\bf\text{$\begin{matrix}  & 0 &  \\ 1 & 1 & 1\end{matrix}$}}} \ar@{->}[rdd]\ar@{-->}[ll]^{\tau}\\
\text{$\begin{matrix}  & 1 &  \\ 0 & 0 & 0 \end{matrix}$} \ar@{->}[rd] & & \fbox{{\bf\text{$\begin{matrix}  & 0 &  \\ 1 & 1 & 0\end{matrix}$}}}\ar@{->}[rd]\ar@{-->}[ll]^{\tau} & & \text{$\begin{matrix}  & 1 &  \\ 0 & 1 & 1\end{matrix}$}\ar@{->}[rd]\ar@{-->}[ll]^{\tau}\\
& \fbox{{\bf\text{$\begin{matrix}  & 1 &  \\ 1 & 1 & 0 \end{matrix}$}}} \ar@{->}[rdd]\ar@{->}[ru]\ar@{->}[ruu]&  & \fbox{{\bf\text{$\begin{matrix}  & 1 &  \\ 1 & 2 & 1 \end{matrix}$}}}\ar@{->}[rdd]\ar@{->}[ru]\ar@{->}[ruu]\ar@{-->}[ll]^{\tau}& & \text{$\begin{matrix}  & 0 &  \\ 0 & 1 & 1\end{matrix}$}\ar@{->}[rdd]\ar@{-->}[ll]^{\tau} \\
\\
& & \fbox{{\bf\text{$\begin{matrix}  & 1 &  \\ 1 & 1 & 1\end{matrix}$}}}\ar@{->}[ruu]& & \text{$\begin{matrix}  & 0 &  \\ 0 & 1 & 0\end{matrix}$}\ar@{->}[ruu]\ar@{-->}[ll]^{\tau} & & \text{$\begin{matrix}  & 0 &  \\ 0 & 0 & 1\end{matrix}$\ar@{-->}[ll]^{\tau}}
}
$
}
\end{center}

We have

\begin{center}
\resizebox{!}{0.45cm}{
$\mathcal{S}^{\ch}_1(\Qu)=\left\lbrace\small\text{$\begin{matrix}  & 0 &  \\ 1 & 0 & 0,\end{matrix}$}\ \ \text{$\begin{matrix}  & 1 &  \\ 1 & 1 & 0,\ \end{matrix}$}\ \ \text{$\begin{matrix}  & 1 &  \\ 1 & 1 & 1,\end{matrix}$}\ \ \text{$\begin{matrix}  & 0 &  \\ 1 & 1 & 0,\end{matrix}$}\ \ \text{$\begin{matrix}  & 1 &  \\ 1 & 1 & 1,\ \end{matrix}$}\ \ \text{$\begin{matrix}  & 0 &  \\ 1 & 1 & 0\end{matrix}$}\oplus \text{$\begin{matrix}  & 1 &  \\ 1 & 1 & 1,\end{matrix}$}\ \
\text{$\begin{matrix}  & 1 &  \\ 1 & 2 & 1,\end{matrix}$}\ \ \text{$\begin{matrix}  & 0 &  \\ 1 & 1 & 1\end{matrix}$}\right\rbrace$.}
\end{center}
We have two chains of maximal length in $\mathcal{S}^{\ch}_i(\Qu)$:

\begin{center}
\resizebox{!}{0.4cm}{
${\small\text{$\begin{matrix}  & 0 &  \\ 1 & 0 & 0\end{matrix}$}\kg^{\ch}\text{$\begin{matrix}  & 1 &  \\ 1 & 1 & 0\end{matrix}$} \kg^{\ch} \text{$\begin{matrix}  & 0 &  \\ 1 & 1 & 0\end{matrix}$}\oplus \text{$\begin{matrix}  & 1 &  \\ 1 & 1 & 1\end{matrix}$} \kg^{\ch} \text{$\begin{matrix}  & 0 &  \\ 1 & 1 & 0\end{matrix}$}\kg^{\ch}\text{$\begin{matrix}  & 1 &  \\ 1 & 2 & 1\end{matrix}$}\kg^{\ch}\text{$\begin{matrix}  & 0 &  \\ 1 & 1 & 1\end{matrix}$}}$
}
\end{center}
and
\begin{center}
\resizebox{!}{0.4cm}{
${\small\text{$\begin{matrix}  & 0 &  \\ 1 & 0 & 0\end{matrix}$}\kg^{\ch}\text{$\begin{matrix}  & 1 &  \\ 1 & 1 & 0\end{matrix}$} \kg^{\ch} \text{$\begin{matrix}  & 0 &  \\ 1 & 1 & 0\end{matrix}$}\oplus \text{$\begin{matrix}  & 1 &  \\ 1 & 1 & 1\end{matrix}$} \kg^{\ch} \text{$\begin{matrix}  & 1 &  \\ 1 & 1 & 1\end{matrix}$}\kg^{\ch}\text{$\begin{matrix}  & 1 &  \\ 1 & 2 & 1\end{matrix}$}\kg^{\ch}\text{$\begin{matrix}  & 0 &  \\ 1 & 1 & 1\end{matrix}$}}$
}

\end{center}

\end{ex}

Fix $i\in I$. For a $k\Qu$--module $M$ and an element $V\in \mathcal{S}_i(\Qu)$ and $V\in \mathcal{S}^{\ch}_i(\Qu)$ respectively define
\begin{equation}\label{eq:F}
{ F_i(M,V)}:=\displaystyle\sum_{B\in\mathcal{P}_i(\Qu); \ B{\unlhd} V} \mu_B(M)-\mu_{\tau B}(M).
\end{equation}
\begin{equation}\label{eq:Fcheck}
{ F^{\ch}_i(M,V)}:=\displaystyle\sum_{B\in\mathcal{P}^{\ch}_i(\Qu); \ V{\unlhd}^{\ch} B} \mu_B(M)-\mu_{\tau^{-1} B}(M).
\end{equation}

For a $\mathbb{C}\Qu$-module $M$, let $V_M$ be a $\unlhd$-maximal element of $\mathcal{S}_i(\Qu)$ such that $$\max_{V\in \mathcal{S}_i(\Qu)}F_i(M,V)=F_i(M,V_M)$$ and let $U_M$ be the direct sum of all $\tau B$ such that $B$ is an element of $\mathcal{P}_i(\Qu)$ with $B \ntrianglelefteq V_M$ and $B$ is minimal with this property. 
We define $B^{h}(\infty)$ to be the set of all isomorphism classes of $\mathbb{C}\Qu$-modules.

\begin{defi} A quiver $\Qu$ is called \textit{cospecial} if $\dim\Hom_{\mathbb{C}\Qu}(S(i),M)\le 1$ for all $i\in I$ and all indecomposable $\mathbb{C}\Qu$-modules $M$. Equivalently, this means that each sink $i\in I$ corrspond to a miniscule weight of the Lie algebra of $\Qu$ (see \cite[Corollary 2.22]{Schu17}). We call $\Qu$ \textit{special} if $\Qu^{*}$ is cospecial. This is equivalent to the fact that no thick vertex is a source of $\Qu$, where a vertex $i\in I$  is called \textit{thick} if there exists an indecomposable $\mathbb{C}\Qu$-module $M = (V,x)$ such that $\dim V_i \geq 2$.
\end{defi}

For a special quiver $\Qu$ it has been shown in \cite[Proposition 6.1]{R97} that the module $U_M$ is a direct summand of $M$ and the following defines a crystal structure on $B^{h}(\infty)$ isomorphic to $B(\infty)$:
$$
\tilde{f}_i(M)=(N\oplus V_M), \text{ where }M=N\oplus U_M$$
$$
\varepsilon_i(M)=F_i(M,V_M), \quad \varphi_i(M)=\varepsilon_i(M)+\wt(M)(h_i)$$
$$
\wt(M)=-\sum_{j\in I}(\underline{\dim}M)_j\alpha_j 
$$
\begin{rem}
\begin{enumerate}
\item The explicit description of $\tilde{e}_i$ can be found in \cite[Proposition 2.19]{Schu17}.
\item The isomorphism between $B^{h}(\infty)$ and $B(\infty)$ implies that the module $V_M$ is unique. An alternative proof of this fact can be  given as in \cite[Lemma 3.15]{Schu17}.
\end{enumerate}
\end{rem}
Moreover, if we set $$\varepsilon_i^{*}(M)=\max_{V\in \mathcal{S}_i^{\vee}(\Qu)} F_i^{\vee}(M,V),$$ and assume that $\Qu$ is cospecial and special, then the embedding of the crystal graph of $B^{h}(\lambda)\cong B(\lambda)$ into  $B^{h}(\infty)$ can be described as (see \cite[Proposition 7.4]{R97}):
$$B^{h}(\lambda)=\{M\in B^{h}(\infty): \varepsilon_i^{*}(M)\le \lambda(h_i)\text{ for all }i\in I\}.$$
\subsection{}The remainder of this section develops combinatorics on the geometric construction of crystal bases recalled in Section~\ref{geometricbinfty} in terms of Auslander-Reiten quivers. Recall the embedding of irreducible components of Nakajima's quiver variety into the irreducible components of Lusztig's quiver variety from \eqref{eq:comp1}. We decribe the image of this embedding. Namely the following result gives a criterion for the irreducible components of $\Lambda(v,\lambda)$ to contain a stable point.

\begin{thm}\label{irrcomst} Let $\Qu$ be cospecial, $M\in \mathbb{C}\Qu-\modd$. The following statements are equivalent.
\begin{itemize}
\item[(i)] $Y^{\lambda}_{M}\in \Ir\mathfrak{L}(v,\lambda)$, i.e. $Y^{\lambda}_{M}\neq \emptyset$.\vspace{0,15cm}

\item[(ii)] $F^{\ch}_i(M,V)\le \lambda(h_i)$ for all $V\in \mathcal{S}^{\ch}_i(\Qu) $ and for all $i \in I$.
\end{itemize}
Moreover we have the equality
\begin{equation}\label{epsilonstar}
\varepsilon_i^{*}(M)=\min_{x\in X_{M}}\dim(\displaystyle\bigcap_{\out(h)=i}\ker x_h)
\end{equation}

\end{thm}

\begin{proof} 
By \eqref{eq:comp1} we have that $Y^{\lambda}_{M}\in \Ir\mathfrak{L}(v,\lambda)$ if and only if $Y^{\lambda}_{M}$
 contains a stable point. Note that stability for a point $(x,t)\in \Lambda(v,\lambda)$ means that 
 $$\Big(\bigcap_{\out(h)=i} \ker x_h\Big) \cap \ker t_i=0\ \ \forall i\in I.$$
This is equivalent to the fact that the restriction of $t_i$ to $\bigcap_{\out(h)=i} \ker x_h$ is injective for all $i\in I$. Hence
$$Y^{\lambda}_{M}\neq \emptyset \Leftrightarrow \exists x\in X_{M}\text{ with }\dim(\displaystyle\bigcap_{\out(h)=i}\ker x_h) \le \lambda(h_i) \ \ \forall i\in I.$$ 

Let $X_{M}^*$ the closure of the conormal bundle to the $G_v$-orbit to $\mathcal{D}(M)$ in $\mathrm{Rep}_V(Q^{*})$ (recall the definition of $\mathcal{D}(M)$ from Section~\ref{section23}). By Lemma~\ref{irrcomplusztig} we have that $X_{M}^*\in B^g(\infty)$. Note that
$$\min_{x\in X_{M}}\dim(\displaystyle\bigcap_{\out(h)=i}\ker x_h)=\varepsilon_i(X_{M}^*).$$
Hence we obtain
\begin{align*}\varepsilon_i(X_{M}^*)\le \lambda(h_i) \text{ for all }i\in I & \Leftrightarrow \varepsilon_i(X_{\mathcal{D}(M)})\leq \lambda(h_i) \text{ for all }i\in I &\\& \Leftrightarrow F_i(\mathcal{D}(M),V)\leq \lambda(h_i) \text{ for all $V\in\mathcal{S}_i(\Qu^{*})$}\text{ and }i\in I,\end{align*}
where the last equivalence follows from \cite[Proposition 4.7]{Schu17} and the fact that $\Qu$ is cospecial if and only if $\Qu^*$ is special. We can identify the elements of $\mathcal{S}^{\ch}_i(\Qu)$ canonically with the elements of $\mathcal{S}_i(\Qu^{*})$ via the standard duality $\mathcal{D}(M)$, since the Auslander-Reiten quiver of $\mathbb{C}{\Qu^*}$ can be obtained by reversing each arrow in the Auslander-Reiten quiver of $\mathbb{C}\Qu$ and interchanging the roles of $\tau$ and $\tau^{-1}$ In particular,
$$\mathcal{D}(V)\in \mathcal{S}^{\ch}_i(\Qu)\Leftrightarrow V\in \mathcal{S}_i(\Qu^{*})$$
Moreover, for $V$ and $V'$ in $\mathcal{S}^{\ch}_i(\Qu)$, we have
$$V\kg V' \text{ if and only if }\mathcal{D}(V)\kg^{\ch}\mathcal{D}(V').$$
This shows that 
\begin{align*}Y^{\lambda}_{M}\neq \emptyset \Leftrightarrow F^{\ch}_i(M,V)\leq \lambda(h_i) \text{ for all $V\in\mathcal{S}^{\ch}_i(\Qu)$}\text{ and }i\in I\end{align*}
since for all $V\in {\mathcal{S}^{\vee}_i(\Qu)}$ we have $F_i(\mathcal{D}(M),\mathcal{D}(V))=F_i^{\ch}(M,V)$. This finishes the proof.
\end{proof}
\subsection{}\label{crystalblambda}
In this section we shall prove the following theorem.
\begin{thm}\label{crystalstructurelambda} 
Let $\Qu$ be special and cospecial. The map $B^{g}(\lambda)\rightarrow B^{h}(\lambda),\ Y^{\lambda}_{M}\mapsto M$ is well-defined and an isomorphism of crystals. We have a commutative diagram
 $$\begin{xy}
  \xymatrix{
      B^{g}(\infty) \ar[rr]   &     & B^{h}(\infty) \\
                              B^{g}(\lambda) \ar[u]^{\mathfrak{i}} \ar[rr]& & B^{h}(\lambda)
\ar[u]^{}    }
\end{xy}$$
\end{thm}

Before we are able to prove Theorem \ref{crystalstructurelambda}, we need some preparatory work to show that the crystal operators are well-defined. 

\begin{lem}\label{egal} Let $M$ be in $\mathbb{C}\Qu-\modd$. Then we have
$$\varepsilon^{*}_i(\tilde{f}_j (M))=\begin{cases}\varepsilon_i^{*}(M),& \text{$i\neq j$}\\
\varepsilon^{*}_i(M) +1 & \text{if $i=j$, } V_M=S(i)\text{ and }\displaystyle\max_{V\in\mathcal{S}^{\ch}_i(\Qu)}F_i^{\ch}(M,V)=F_i^{\ch}(M,S(i))\\
\varepsilon^{*}_i(M) & \text{else.} \end{cases}$$
\end{lem}

\begin{proof}The case $i\neq j$ is clear by \eqref{epsilonstar} and the definition of  $\tilde{f}_i$ (see Section~\ref{gc1}).

Let $i=j$. We consider two cases.
\vspace{0,2cm}

\textit{Case 1:} Assume that $\tilde{f}_i(M)=(M\oplus S(i))$. Hence
$$F_i^{\ch}(M,V)=F_i^{\ch}(M\oplus S(i),V) \text{ for all } V\in \mathcal{S}^{\ch}_i(\Qu)\backslash \{S(i)\}$$
and we get  $$F_i^{\ch}(M,S(i))=F_i^{\ch}(M\oplus S(i),S(i))+1.$$
We conclude for this case
$$\varepsilon^{*}_i(\tilde{f}_i (M))=\begin{cases} \varepsilon^{*}_i(M) +1  & \text{ if } \displaystyle\max_{V\in\mathcal{S}^{\ch}_i(\Qu)}F_i^{\ch}(M,V)=F_i^{\ch}(M,S(i)) \\
\varepsilon^{*}_i(M) & \text{ else.} \end{cases}$$
\vspace{0,2cm}

\textit{Case 2:} In this case we suppose that $\tilde{f}_i (M)=N$ with $N \ncong M\oplus S(i)$. Then $N\cong M'\oplus V_M$ where $M=M'\oplus U_M$ with $U_M$ and $V_M$ as in Section \ref{section41}. Let $B$ be an indecomposable $\mathbb{C}\Qu$-module in $\mathcal{P}_i(\Qu)\cap \mathcal{P}^{\ch}_i(\Qu)$. Then we have the following homomorphisms:
$$B\twoheadrightarrow S(i) \hookrightarrow B$$
which implies $\mathcal{P}_i(\Qu)\cap \mathcal{P}^{\ch}_i(\Qu)=\{S(i)\}$. So there is no indecomposable direct summand of $V_M$ in $\mathcal{P}^{\ch}_i(\Qu)$ which gives
$$\varepsilon^{*}_i(\tilde{f}_i(M))\le \varepsilon^{*}_i(M).$$
Assume that there is an indecomposable direct summand $C$ of $U_M$ in $\mathcal{P}^{\ch}_i(\Qu)$. Then, from the definition, there is a $B\in \mathcal{P}_i(\Qu)$ such that $C=\tau B$ and we have homomorphisms
$$B\twoheadrightarrow S(i) \hookrightarrow C=\tau B.$$
Hence $\Hom_{\mathbb{C}\Qu}(B,\tau B)\ne 0$, which is a contradiction. An analog argument shows that there cannot be an indecomposable direct summand  $B$ of $V_M$ such that $\tau^{-1}B\in \mathcal{P}^{\ch}_i(\Qu)$. This yields in this case
$$\varepsilon^{*}_i(\tilde{f}_i(M))=\varepsilon^{*}_i(M).$$
\end{proof}
\begin{prop}\label{keine} Let $\lambda\in P^+$ and $M$ be an element of $B^{h}(\lambda)$. Then for all $i\in I$:
$$\exists j\in I: \varepsilon^{*}_j(\tilde{f}_i(M))> \lambda(h_j) \Longleftrightarrow \varphi_i(X_{M})=0.$$
\end{prop}

\begin{proof} Since $\varepsilon^{*}_i(M)\leq \lambda(h_i)$ for all $i\in I$, it follows from Lemma~\ref{egal} that we only need to show the equivalence
$$\varepsilon^{*}_i(\tilde{f}_i(M))> \lambda(h_i) \Longleftrightarrow \varphi_i(X_{M})=0.$$
In the case of $\varepsilon^{*}_i(\tilde{f}_i(M))> \lambda(h_i)$ we obtain with Lemma~\ref{egal} that
$$\varepsilon^{*}_i(M)=\lambda(h_i),\ \ \ \varepsilon_i(M)=F_i(M,S(i)).$$
Since $\Qu$ is special and cospecial, we have
$$\dim\Hom_{\mathbb{C}\Qu}(M,S(i))=\sum_{B\in \mathcal{P}_i(\Qu)}\mu_B(M)\dim\Hom_{\mathbb{C}\Qu}(B,S(i))=\sum_{B\in \mathcal{P}_i(\Qu)}\mu_B(M).$$
$$\dim\Hom_{\mathbb{C}\Qu}(S(i),M)=\sum_{B\in \mathcal{P}^{\vee}_i(\Qu)}\mu_B(M)\dim\Hom_{\mathbb{C}\Qu}(S(i),B)=\sum_{B\in \mathcal{P}^{\vee}_i(\Qu)}\mu_B(M).$$
This implies 
\begin{align*}F_i(M,S(i))+F^{\vee}_i(M,S(i))&=\dim\Hom_{\mathbb{C}\Qu}(M,S(i))-\dim\Hom_{\mathbb{C}\Qu}(\tau^{-1}M,S(i))&\\&
+\dim\Hom_{\mathbb{C}\Qu}(S(i),M)-\dim\Hom_{\mathbb{C}\Qu}(S(i),\tau M)&\\&=\langle M,S(i)\rangle _R+\langle S(i),M\rangle _R,
\end{align*}
where the last equation follow from the Auslander-Reiten formulas (see \cite[Corollary 2.14]{ASS06}). Hence we have 
\begin{align*}
\varphi_i(X_{M})&=\varepsilon_i(X_{M})+\wt(X_{M})(h_i)\\
&=F_i(M,S(i))+\lambda(h_i)-(M,S(i))_R\\
&=F_i(M,S(i))+\lambda(h_i)-\langle M,S(i)\rangle _R-\langle S(i),M\rangle _R\\
&=F_i(M,S(i))+\lambda(h_i)-F_i(M,S(i))-F_i^{\ch}(M,S(i))=0.
\end{align*}
where the second equality follows from the crystal isomorphism in \cite[Theorem 3.26]{Schu17}. 
Conversely, assume that $\varphi_i(X_{M})=0,$
i.e.
$$0=\varepsilon_i(X_{M})+\wt(X_{M})(h_i)=\varepsilon_i(X_{M})+\lambda(h_i)-F_i(M,S(i))-F_i^{\ch}(M,S(i)).$$
Since $M\in B^{h}(\lambda)$ we have $F_i^{\ch}(M,S(i))\leq \lambda(h_i)$ and $F_i(M,S(i))\leq \varepsilon_i(X_{M})$. Thus
$$\varepsilon_i(X_{M})=F_i(M,S(i)),\ \ \lambda(h_i)=F_i^{\ch}(M,S(i))=\varepsilon_i^{*}(M).$$
Using Lemma \ref{egal}, we get
$\varepsilon^{*}_i(\tilde{f}_i(M))> \varepsilon_i^{*}(M) =\lambda(h_i).$
\end{proof}

\textit{Proof of Theorem \ref{crystalstructurelambda}}
 From  Lemma~\ref{egal} and Proposition~\ref{keine} we obtain that the restriction of the crystal isomorphism $B^{g}(\infty)\rightarrow B^{h}(\infty),\ X_M\mapsto M$ (see \cite[Theorem 3.26]{Schu17}) induces an isomorphism of crystals $B^{g}(\lambda)\rightarrow B^{h}(\lambda).$ 

\section{Applications to affine crystals and the promotion operator}\label{KRcryssec}
In this section we consider the type $A_n$ quiver $\Qu$ with index set $I=\{1,\dots,n\}$ and orientation 
$$
\nonumber
\xygraph{
[]
!{<0pt,0pt>;<20pt,0pt>:}
*\cir<2pt>{}
!{\save -<0pt,6pt>*\txt{$_n$}  \restore}
-@{{}->} [l]
*\cir<2pt>{}
!{\save -<0pt,6pt>*\txt{$_{n-1}$}  \restore}
-@{{}->} [l]
*\cir<2pt>{}
!{\save -<0pt,6pt>*\txt{$_{n-2}$}  \restore}
-@{{}->} [l]
*{ \;  \dots \; }
-@{{}->} [l]
*\cir<2pt>{}
!{\save -<0pt,6pt>*\txt{$_{2}$}  \restore}
-@{{}->} [l]
*\cir<2pt>{}
!{\save -<0pt,6pt>*\txt{$_{1}$}  \restore}
}
$$
\subsection{} Our aim is to describe the promotion operator on the crystel $B^{h}(\lambda)$ and hence on $B^{g}(\lambda)$ for a fixed rectangular weight $\lambda=m\varpi_j$, $j\in I$. This will allow us to realize Kirillov-Reshetikhin crystals geometrically. 

For an indecomposable module $M(r,s)$ (corresponding to the root $\alpha_r+\cdots+\alpha_s$) we abbreviate the multiplicity of $M(r,s)$ in a $\mathbb{C}\Qu$-module $M$ simply by $\mu_{r,s}(M)$. We fix for the rest of this section an element $M\in B^{h}(\lambda)$. Note that the choice of the orientation gives 
\begin{equation}\label{chor}F_i^{\vee}(M,M(i,s))=\sum_{k=s}^n \mu_{i,k}(M)-\sum_{k=s}^{n-1} \mu_{i+1,k+1}(M).\end{equation}
Since $F_i^{\vee}(M,V)\leq \lambda(h_i)$ for all $i\in I$ and $V\in\mathcal{S}_i^{\vee}(\Qu)$ we get that 
$$k_i:=m-\dim\Hom_{\mathbb{C}\Qu}(S(i),M)\geq 0.$$ Moreover, since $\lambda$ is rectangular, we have
$$\mu_{r,s}(M)=0,\ \ \text{for all $r>j$ or ($r\leq j$ and $s>n-j+r$)}.$$
So we can think of an element $M\in B^{h}(\lambda)$ as an array $M=(\mu_{r,s}(M))$ with $1\leq r\leq j$ and $r\leq s\leq n-j+r$ in the Auslander-Reiten quiver $\Gamma_{\Qu}$. We associate to such an array its extended array
$$M^{\mathrm{ext}}=(\mu_{r,s}(M^{\mathrm{ext}})),\ \ 1\leq r\leq j,\ r\leq s\leq n+1-j+r$$
defined by 
$$\mu_{r,s}(M^{\mathrm{ext}})=\begin{cases} \mu_{r,s-1}(M) ,& \text{ if $r\neq s$}\\
k_r ,& \text{ if $r=s$}
\end{cases}$$
and view $M^{\mathrm{ext}}$ as a $\mathbb{C} \widetilde{\Qu}$-module, where $\widetilde{\Qu}$ is the $A_{n+1}$ quiver with same standard orientation.
\begin{ex} Let $n=3$ and $j=2$. Below is the array of $M^{\mathrm{ext}}$ and the array of $M$ is highlighted in blue (we abbreviate $\mu_{r,s}(M)=\mu_{r,s}$).
$$
\xymatrixrowsep{0.02in}
\xymatrixcolsep{0.02in}
\xymatrix{
& &{\color{blue}\mu_{1,2}}&  &{\color{blue}\mu_{2,3}}\\
& {\color{blue}\mu_{1,1}} & &{\color{blue} \mu_{2,2}}  & \\
k_1&  & k_2& &\\
}
$$
where we have $k_1=m-\mu_{1,1}- \mu_{1,2}$ and $k_2=m-\mu_{2,2}- \mu_{2,3}$.
\end{ex}
\subsection{}Fix an arbitrary $\mathbb{C} \widetilde{\Qu}$-module $N$ and let $W_{N}$ the $\unlhd^{\vee}$-maximal element of $\mathcal{S}^{\vee}_i(\widetilde{\Qu})$ such that $$\max_{V\in \mathcal{S}^{\vee}_i(\tilde{\Qu})}F^{\vee}_i(N,V)=F^{\vee}_i(N,W_{N})$$ and let $E_{N}=\tau^{-1} B$ where $B$ is an element of $\mathcal{P}^{\vee}_i(\widetilde{\Qu})$ with $W_{N} \ntrianglelefteq^{\vee} B$ and $B$ is maximal with this property. Since $\mathcal{S}_i^{\vee}(\widetilde{\Qu})=\{M(i,i), M(i,i+1),\dots, M(i,n+1)\}$ we have that $W_{N}$ and $E_{N}$ are indecomposable (if the latter exists). With other words 
$$W_{N}=M(i,s_0),\ \ \ E_{N}=M(i+1,s_0),$$
where $$s_0=\max\{i\leq s\leq n+1: F_i^{\vee}(M,M(i,s)) \text{ is maximal}\}.$$
\begin{defi}\label{defpr}
 \begin{enumerate}
\item Let $N$ be a $\mathbb{C} \widetilde{\Qu}$-module. We define operators $T_i$ for $i<j$ and $\mathrm{sh}_j$ in order to obtain another $\mathbb{C} \widetilde{\Qu}$-module as follows.
Let $T_i(N)=N$ (resp. $\mathrm{sh}_{j}(N)=N$) if $W_{N}$ (resp. $M(j,n+1)$) is not a summand  of $N$ and otherwise we set
\begin{align*}&T_i(N)=
N'\oplus E_{N},\ \text{where $N=N'\oplus W_{N}$},&\\&
\mathrm{sh}_{j}(N)=N'', \text{ where $N=N''\oplus M(j,n+1).$}\end{align*}
We define further 
$$\widetilde{\mathrm{pr}}(N)= (T_1\circ \cdots \circ T_{j-1}\circ \mathrm{sh}_{j})^{\mu_{j,n+1}(N)}(N).$$
\item We let $\mathrm{pr}(M)$ to be the $\mathbb{C}\Qu$-module determined by the array 
$$\mu_{r,s}(\mathrm{pr}(M))=\mu_{r,s}(\widetilde{\mathrm{pr}}(M^{\mathrm{ext}}))\ \ \ 1\leq r\leq j,\ \ r\leq s\leq n-j+r.$$
\end{enumerate}
\end{defi}
\begin{ex}\label{expr}
Let $n=m=j=3$ and consider the element $M\in B^{h}(\lambda)$ with array
$$
\xymatrixrowsep{0.02in}
\xymatrixcolsep{0.02in}
M=\xymatrix{
& &1&  &1&& 2\\
& 0 & &1  && 0& \\
1&  & 1 & & 1&&\\
}
$$
We get by applying our operators 
$$
\xymatrixrowsep{0.02in}
\xymatrixcolsep{0.02in}
M^{\mathrm{ext}}=\xymatrix{
&& &1&  &1&& 2\\
&& 0 & &1  && 0& \\
&1&  & 1 & & 1&&\\
1&  & 0 & & 0&&\\
}\ \ \ \ \xrightarrow{T_2\circ \mathrm{sh}_3}\xymatrix{
&& &1&  &1&& 1\\
&& 0 & &0  && 0& \\
&1&  & 1 & & 2&&\\
1&  & 0 & & 0&&\\
}
\ \ \ \ \xrightarrow{T_1}\xymatrix{
&& &1&  &1&& 1\\
&& 0 & &0  && 0& \\
&1&  & 1 & & 2&&\\
0&  & 0 & & 0&&\\
}
$$
$$
\xymatrixrowsep{0.02in}
\xymatrixcolsep{0.02in}
\xymatrix{\xrightarrow{\mathrm{sh}_3}
&& &1&  &1&& 0\\
&& 0 & &0  && 0& \\
&1&  & 1 & & 2&&\\
0&  & 0 & & 0&&\\
}\ \ \ \ \xrightarrow{T_2}\xymatrix{
&& &1&  &0&& 0\\
&& 0 & &0  && 1& \\
&1&  & 1 & & 2&&\\
0&  & 0 & & 0&&\\
}
\ \ \ \ \xrightarrow{T_1}\xymatrix{
&& &0&  &0&& 0\\
&& 0 & &1  && 1& \\
&1&  & 1 & & 2&&\\
0&  & 0 & & 0&&\\
}
$$
Hence we get 
$$
\xymatrixrowsep{0.02in}
\xymatrixcolsep{0.02in}\mathrm{pr}(M)=\xymatrix{
&& 0 & &1  && 1& \\
&1&  & 1 & & 2&&\\
0&  & 0 & & 0&&\\
}$$
\end{ex}

\subsection{}\label{section53} The promotion operator (see \cite{S72} for details) is the analogue of the cyclic Dynkin diagram automorphism on the level of crystals. On the set of all semi-standard Young tableaux $\mathrm{SSYT}(\lambda)$ of shape $\lambda$ over the alphabet $1\prec 2\cdots \prec n+1$ the promotion operator can be described as follows. Let $T$ be a Young tableaux, then we get $\mathrm{pr}(T)$ by removing all letters $(n+1)$, adding 1 to each letter in the remaining tableaux, using jeu-de-taquin to slide all letters up and finally filling the holes with 1's. Combining \cite[Theorem 6.4]{Sa06} and Theorem~\ref{crystalstructurelambda} we get a classical crystal isomorphism (which is known explicitly only for the standard orientation)
$$\varphi: B^{h}(\lambda)\rightarrow \mathrm{SSYT}(\lambda),\ M\mapsto \varphi(M)$$
where the semi-standard tableaux $\varphi(M)$ for $M\in B^{h}(\lambda)$ is determined by 
\begin{equation}\label{tabdef}\mu_{r,s}(M)=\text{\# (s+1) in row r of $\varphi(M)$},\ \ 1\leq r\leq j,\ \ r\leq s\leq n-j+r.\end{equation}
We claim that $\mathrm{pr}$ from Definition~\ref{defpr} is the promotion operator on $B^{h}(\lambda)$ and our strategy is to show that it commutes with the promotion operator on $\mathrm{SSYT}(\lambda)$ under the isomorphism $\varphi$. It is possible to define a tableau $\varphi(M')$ using \eqref{tabdef} for any $\mathbb{C}\Qu$-module $M'$. Obviously we have
\begin{equation}\label{refw2}
M'\in B^{h}(\lambda) \Longleftrightarrow \varphi(M')\in \mathrm{SSYT}(\lambda)
\end{equation}

\begin{thm}\label{isopr}
Let $\lambda=m\varpi_j$ be a rectangular weight of type $A_n$. Then we have a commutative diagram 
$$\begin{xy}
  \xymatrix{
      B^{h}(\lambda)\ar[d]^{\mathrm{pr}} \ar[rr]^{\varphi}   &     & \mathrm{SSYT}(\lambda) \ar[d]^{\mathrm{pr}}\\
                              B^{h}(\lambda)  \ar[rr]^{\varphi} & & \mathrm{SSYT}(\lambda)   }
\end{xy}$$
\begin{proof}
Let $M\in B^{h}(\lambda)$ and let $\mathrm{pr}(M)$ the $\mathbb{C}\Qu$-module as described in Definition~\ref{defpr}. If we can show that $\mathrm{pr}(\varphi(M))=\varphi(\mathrm{pr}(M))$ we get at once $\mathrm{pr}(M)\in B^{h}(\lambda)$ and the commutativity by \eqref{refw2}. Note that $M^{\mathrm{ext}}$ also encodes the tableaux $\varphi(M)$ via
$$\mu_{r,s}(M^{\mathrm{ext}})=\text{\# s in row $r$'s of $\varphi(M)$},\ \ 1\leq r\leq j,\ \ r\leq s\leq n-j+r.$$ In the remaining part of the proof we will show that $\widetilde{\mathrm{pr}}(M^{\mathrm{ext}})$ encodes the skew-tableaux of shape $\lambda\backslash \mu$, $\mu=(\mu_{j,n}(M),0,\dots)$, which is obtained from $\varphi(M)$ as follows: remove all letters $n+1$ and slide the boxes up using jeu-de-taquin. This would finish the proof by construction (see Definition~\ref{defpr}(2)). So assume that the last two rows of $\varphi(M)$ are given by
$$\ytableausetup
 {mathmode, boxsize=2.4em}\begin{ytableau}
a_1 & a_2 &\cdots & \tiny{a_{m-1}}  &a_m \\
b_1 & b_2  &\cdots  & \tiny{b_{m-1}} &b_m \\
\end{ytableau} 
$$
with $b_m=b_{m-1}=\cdots=b_{m-\mu_{j,n}(M)+1}=n+1$. So the tableaux corresponding to $\mathrm{sh}_j(M^{\mathrm{ext}})$ is given by 
$$\ytableausetup
 {mathmode, boxsize=2.5em}\begin{ytableau}
a_1 & a_2 &\cdots & \tiny{a_{m-1}}  &a_m \\
b_1 & b_2  &\cdots  & \tiny{b_{m-1}} & \\
\end{ytableau} 
$$
If we slide the empty box to the $(j-1)$-th row we have to move an entry $a_p$ to the $j$-th row. In order to figure out what $p$ is we have to consider the sums:
\begin{align*}A_r:=\sum_{k=r}^{n+1}\mu_{j,k}(M^{\mathrm{ext}})-\sum_{k=r-1}^{n}\mu_{j-1,k}(M^{\mathrm{ext}})\geq 0\ \ \ j\leq r\leq n+1.
\end{align*}
Then we have $p=\max\{j\leq r\leq n+1: -A_r \text{ is maximal}\}$. By the definition of $T_{j-1}$ this exactly means that the tableaux corresponding to $T_{j-1}\mathrm{sh}_j(M^{\mathrm{ext}})$ is given by 
$$\ytableausetup
 {mathmode, boxsize=2.5em}\begin{ytableau}
a_1 & a_2 &\cdots &a_{p-1}& a_{p+1}& \cdots &a_{m}  &\\
b_1 & b_2 &\cdots &b_{p-1}& a_{p}&b_{p}& \cdots &b_{m-1}  \\
\end{ytableau} 
$$
Hence $T_{j-1}$ slides the empty box in row $j$ to the $(j-1)$-th row following the rules of jeu-de-taquin. Now it is clear that sliding the empty box in row $(j-1)$-th to the top means to apply the operators $T_{j-2},\dots,T_1$ to $T_{j-1}\mathrm{sh}_j(M^{\mathrm{ext}})$. Repeating the above steps yields the claim.
\end{proof}
\end{thm}

\begin{ex}We consider the same situation as in Example~\ref{expr}. Then $M$ corresponds to the Young tableaux 
$$\ytableausetup
 {mathmode, boxsize=1.3em}\varphi(M)= \begin{ytableau}
1 & 2 & 4 \\
3 & 4 & 5 \\ 
4 & 6 & 6 \\
\end{ytableau} \ \ 
\xrightarrow{}\ \ \mathrm{pr}(\varphi(M))=\ytableausetup
 {mathmode, boxsize=1.3em}\begin{ytableau}
1 & 1 & 3 \\
2 & 4 & 5 \\ 
5 & 5 & 6 \\
\end{ytableau} 
$$
By the calculations in Example~\ref{expr} we see that $\mathrm{pr}(\varphi(M))$ coincides with $\varphi(\mathrm{pr}(M))$.
\end{ex}

\subsection{}Combinatorial descriptions of Kirillov-Reshetikhin crystals of type $A^{(1)}_n$ were provided for example in \cite{K12} and \cite{S02}, where the affine crystal structure in the latter work is given without using the promotion operator. The affine Kashiwara operators are given by \begin{equation}\label{affka}\tilde{f}_{0}:=\mathrm{pr}^{n}\circ \tilde{f}_{1}\circ \mathrm{pr}, \mbox{ and } \tilde{e}_{0}:=\mathrm{pr}^{n}\circ \tilde{e}_{1}\circ \mathrm{pr}.\end{equation}
\begin{cor}\label{geoKR}Let $\lambda=m\varpi_j$ be a rectangular weight of type $A_n$. The following operation gives $B^{g}(\lambda)$ the structure of an affine crystal isomorphic to the Kirillov-Reshetikhin crystal 
$$\tilde{f}_0 Y_{[M]}^{\lambda}=\begin{cases} Y_{[\mathrm{pr}(M)^n\circ \tilde{f}_1\circ \mathrm{pr}(M)]}^{\lambda},& \text{ if $k_1<\mu_{j,n}(M)+\mu_{1,1}(\mathrm{pr}(M))$}\\
0,& \text{ otherwise}.\end{cases}$$
\begin{proof}
We only have to check that $\tilde{f}_1$ acts on $\mathrm{pr}(M)$ if and only if $k_1<\mu_{j,n}(M)+\mu_{1,1}(\mathrm{pr}(M))$. The rest follows from Theorem~\ref{isopr} and Theorem~\ref{crystalstructurelambda}. Clearly $\tilde{f}_1$ acts if and only if 
$$\dim\Hom_{\mathbb{C}\Qu}(S(1),\mathrm{pr}(M))<\dim\Hom_{\mathbb{C}\Qu}(S(2),\mathrm{pr}(M)).$$
Obviously in the process of obtaining $\mathrm{pr}(M)$ we subtract $\mu_{j,n}(M)$ entries in the first column of $M$ (viewed in the array of $M$ as in Example~\ref{expr}), i.e. $\dim\Hom_{\mathbb{C}\Qu}(S(1),\mathrm{pr}(M))=m-\mu_{j,n}(M)$. Similarly we get a
$$\dim\Hom_{\mathbb{C}\Qu}(S(2),\mathrm{pr}(M))=m-\mu_{j,n}(M)+w,$$
where $w$ is the number of times we add an entry to the second column in the process of obtaining $\mathrm{pr}(M)$. Since $w=\mu_{j,n}(M)+\mu_{1,1}(\mathrm{pr}(M))-k_1$ we get the desired result.
\end{proof}
\end{cor}

\bibliographystyle{plain}
\bibliography{bibfile}
\end{document}